\documentclass[a4paper,12pt]{article}

\usepackage{color}
\usepackage{amssymb, amsmath, amsthm}
\usepackage{graphicx}
\usepackage{url} 

\newtheorem{theorem}{Theorem}
\newtheorem{proposition}[theorem]{Proposition}
\newtheorem{corollary}[theorem]{Corollary}
\newtheorem{lemma}[theorem]{Lemma}
\theoremstyle{definition}
\newtheorem{example}[theorem]{Example}
\theoremstyle{remark}
\newtheorem{remark}[theorem]{Remark}

\def\pd#1{\partial_{#1}}
\def\rank#1{{\rm rank}\, #1}

\def\Prob{\mathbf P}
\def\ch#1{{\rm Char}(#1)}
\def\ord#1{{\rm ord}(#1)}

\def\aff{{\rm aff} }
\def\comment#1{#1}

\title{
Holonomic modules associated with multivariate normal probabilities of polyhedra
}
\author{Tamio Koyama}
\date{}

\begin{document}
\maketitle
\begin{abstract}
The probability content of a convex polyhedron 
with a multivariate normal distribution
can be regarded as a real analytic function.
We give a system of linear partial differential equations 
with polynomial coefficients
for the function and show that the system 
induces a holonomic module.
The rank of the holonomic module is equal to
the number of nonempty faces of the convex polyhedron, 
and we provide an explicit Pfaffian equation (an integrable connection) that is 
associated with the holonomic module.
These are generalizations of results for the Schl\"afli function that were
given by Aomoto.
\\\noindent
MSC classes: 16S32, 62H10
\end{abstract}
\section{Introduction}
A convex polyhedron $P$ is the intersection of half-spaces of the 
$d$-dimensional Euclidean space $\mathbf R^d$.
We are interested in numerical calculation of the probability ${\bf P}(X\in P)$,
where $X$ is a random vector distributed as a $d$-dimensional 
normal distribution with mean $\mu$ and covariance matrix $\Sigma$.
When $P$ is the orthant 
$\{x\in \mathbf R^d:\, x_i\geq 0, \, 1\leq i\leq d\}$ in $\mathbf R^d$, 
the probability is called the orthant probability
and can be regarded as a function of $\mu$ and $\Sigma$.
By the inclusion--exclusion identity given in \cite{edelsbrunner} and \cite{naiman-wynn}, 
the probability content of the convex polyhedron can be written as 
a linear combination of the orthant probabilities.
For this reason, the literature includes many discussions of methods 
for evaluating the orthant probabilities.
For example, 
\cite{mkh} proposed a method based on recursive integration, 
and \cite{genz1992} and \cite{mvtnorm} proposed 
the use of the randomized quasi-Monte Carlo procedure.
For details, see \cite{genz-bretz2009}.
In \cite{kuriki-miwa-hayter}, the probability content of a general convex polyhedron
was evaluated by calculating the orthant probabilities.

The motivation of our study is to evaluate the probability content of a convex
polyhedron by a completely different and novel approach that uses 
the holonomic gradient method (HGM) proposed in \cite{n3ost2}.
The HGM, which is based on the theory and algorithms of $D$-modules, 
is a method for numerically calculating definite integrals.
It can be applied to a broad class of problems.
In fact, various applications of HGM have been proposed, for example,  
\cite{hnt2}, \cite{so3}, and \cite{A-prob}.
In order to apply HGM, we need to regard the probability $\Prob(X\in P)$ 
as a function and provide an explicit Pfaffian equation for it.
In the case where $P$ is the orthant and $\mu=0$, the orthant probability
is a function of $\Sigma$ and Schl\"afli gave a recurrence formula for 
it in \cite{schlafli}.
In \cite{plackett}, Plackett generalized Schl\"afli's result
for the case where $P$ is the orthant and $\mu\neq 0$. 
In \cite{koyama-takemura1}, we provided a holonomic system and a Pfaffian equation
that is associated with the orthant probability. Our Pfaffian equation corresponds with the reduction formula 
in \cite{plackett} and \cite{gassmann}.
In this paper, we generalize our previous results \cite{koyama-takemura1} and 
give a recurrence formula as a Pfaffian equation 
for the case of a general convex polyhedron.

Let $\Sigma = AA^\top$ be a Cholesky decomposition of the covariance matrix $\Sigma$, 
and let $Y$ be a random vector that is $d$-variate normally distributed with mean vector $0$ and for which the covariance matrix is the identity matrix.
Then, we have ${\bf P}(X\in P)={\bf P}(AY+\mu\in P)$, and 
the set $\{y\in \mathbf R^d : Ay+\mu \in P\}$ is also a convex polyhedron.
Hence, it is enough to consider the case in which the mean vector of $X$ is $0$ and the covariance matrix of $X$ is 
the identity matrix.
Under this assumption, 
the probability ${\bf P}(X\in P)$ can be written as 
\[
{\bf P}(X\in P)=
\frac{1}{(2\pi)^{d/2}}
\int_{x\in P} \exp(-\frac{1}{2}\sum_{i=1}^dx_i^2)dx_1\cdots dx_d.
\]
The polyhedron $P$ can be written as 
\begin{equation}\label{calP}
P=\{x\in \mathbf R^d :\, \sum_{i=1}^d\tilde a_{ij}x_i+\tilde b_j\geq 0, \, 
1\leq j\leq n\},
\end{equation}
where $\tilde a_{ij}$ and $\tilde b_j$ are real numbers.
We wish to study this integral with the HGM, 
and as a first step, 
we will assume that the convex polyhedron is in the ``general position;''
the precise definition of ``general position'' 
will be given in Section \ref{sec3}.

Let $a$ be a $d\times n$ matrix, and let $b$ be a vector with length $n$.
We are interested in the analytic properties of the function 
\begin{equation}\label{p0}
\varphi(a, b)=
\int_{\mathbf R^d} 
\exp(-\frac{1}{2}\sum_{i=1}^dx_i^2)
\prod_{j=1}^n H\left(\sum_{i=1}^d a_{ij}x_i+b_j\right)
dx_1\cdots dx_d,
\end{equation}
which is defined on a neighborhood of 
$\tilde a =(\tilde a_{ij}), \tilde b =(\tilde b_j)$.
Here, we denote by $H$ the Heaviside function.
Note that this function is 
an interesting specialization
of the one studied by Aomoto in \cite{aomoto1982a} and 
Aomoto, Kita, Orlik, and  Terao in \cite{aomoto-kita-orlik-terao}.
For the meaning of the specialization, see Remark \ref{rem:1}, below.
In this paper, we provide a holonomic system and a Pfaffian equation associated with 
this function. The Pfaffian equation is required by the HGM for $\varphi(a, b)$.
In order to explicitly provide the holonomic system, we decompose the function 
by the inclusion--exclusion identity associated with the polyhedron $P$.
We also show that the holonomic rank of the system is equal to the number of 
nonempty faces of the polyhedron $P$.
This Pfaffian equation is a generalization of the recursion formula 
given by Plackett \cite{plackett}.
In addition, the singular locus of the Pfaffian equation 
is compatible with that of the Schl\"afli function given in \cite{aomoto1977}.

This paper is constructed as follows.
In Section \ref{sec2}, we 
will give a brief explanation of 
holonomic modules and Pfaffian equations.
In  Section \ref{sec3}, we provide an analytic continuation of 
the function $\varphi(a, b)$.
In section \ref{sec:6}, we prove an existence of an open neiborhood of a 
given point in general position.
In Section \ref{sec4}, we give 
a system of linear partial differential equations with polynomial coefficients
for the function $\varphi(a, b)$
and show that the system induces a holonomic module.
In Section \ref{sec5}, we show that the holonomic rank of the module is equal to the number of nonempty faces of $P$
and explicitly provide the Pfaffian equation associated with the module.

\section{Holonomic module and Pfaffian equation}
\label{sec2}
Before starting the main discussion, we briefly review 
holonomic modules and Pfaffian equations.
For a comprehensive presentation, see \cite{n3ost2} 
and the references cited therein.
We denote 
by $D_x=\mathbf C\langle x_i, \pd{x_i}: 1\leq i\leq n\rangle $
the ring of differential operators of $n$ variables $x_1, \cdots, x_n$ 
with polynomial coefficients.
Here, we put $\pd{x_i}=\partial /\partial x_i$. 
Let us consider a system of linear partial differential equations 
\begin{equation}\label{diff-eq}
\sum_{j=1}^m P_{kj}g_j = 0 \quad (P_{kj}\in D_x, \, 1\leq k\leq m')
\end{equation}
for unknown functions $g_1, \dots, g_m$. 
Let $(D_x)^m$ be the free $D_x$-module with the basis $\{g_1, \dots, g_m\}$, and let 
$N$ be a $D_x$-submodule of $(D_x)^m$ generated by 
$P_k=\sum_{j=1}^mP_{kj}g_j\in (D_x)^m\, (1\leq k\leq m')$.
Note that the basis $\{g_1, \dots, g_m\}$ is an arbitrary set and $g_i$ is not a function. 
We denote by $M$ the quotient module $(D_x)^m/N$.

The set consisting of the holomorphic functions on a domain 
$U\subset \mathbf C^n$ forms a left $D_x$-module $\mathcal O(U)$.
For a morphism $\varphi:M\rightarrow {\cal O}(U)$ of left $D_x$-modules,
the functions $\varphi(g_1), \dots, \varphi(g_m)$ satisfy 
system \eqref{diff-eq}.
For this reason, we call a vector-valued function $(g'_1, \dots, g'_m)$ on $U$ 
a solution of $M$ 
when there is a morphism of $D_x$-modules $\varphi:M\rightarrow{\cal O}(U)$ 
such that $\varphi(g_j)=g'_j\, (1\leq j\leq m)$.

By the theory of the Gr\"obner basis in Weyl algebra, 
the {\it characteristic variety} $\ch{M}$ of $M$ can be computed
explicitly.
For details, see \cite{oaku1994}. 
According to the Bernstein inequality, 
the Krull dimension of the characteristic variety is not less than $n$
(see, e.g., \cite{bjork}).
When the dimension of $\ch{M}$ is equal to $n$, 
the $D_x$-module $M$ is said to be {\it holonomic}.
When the system of differential equations \eqref{diff-eq} induces 
a holonomic $D_x$-module, we call \eqref{diff-eq} a holonomic system.

We denote by $R_x$ 
the ring of differential operators of $n$ variables $x_1, \cdots, x_n$ 
with rational function coefficients.
The left $D_x$-module $\mathbf C(x)\otimes_{\mathbf C[x]}M$ is a left $R_x$-module,
where $\mathbf C(x)$ is the field of rational functions.
When the module $M$ is holonomic, 
$\mathbf C(x)\otimes_{\mathbf C[x]}M$ 
as a linear space over $\mathbf C(x)$ has finite dimension.
This value is called the holonomic rank of $M$, and we denote it by $\rank{M}$.
Let $r$ be the holonomic rank of $M$, 
and let $\{f_1, \dots, f_r\}$ be a basis of $\mathbf C(x)\otimes_{\mathbf C[x]}M$ as 
a linear space over $\mathbf C(x)$.
Then, there exist rational functions $c_{ijk}\in \mathbf C(x)(1\leq i, j, k\leq r)$ 
such that 
\begin{equation}\label{def-pfaff}
\pd{x_i}f_j = \sum_{k=1}^r c_{ijk}f_k \quad (1\leq i, j\leq r)
\end{equation}
in $\mathbf C(x)\otimes_{\mathbf C[x]}M$.
Moreover, the matrices $c_i=(c_{ijk})_{j, k=1}^r\, (1\leq i\leq r)$ satisfy
the integrability condition
\[
\frac{\partial c_i}{\partial x_j} + c_ic_j = \frac{\partial c_j}{\partial x_i} + c_jc_i 
\quad (1\leq i, j \leq r).
\]
We call equation \eqref{def-pfaff} a Pfaffian equation associated with 
the holonomic module $M$.
The union of the zero sets of the denominators of the elements of $c_i$'s 
is called the singular locus of the Pfaffian equation.
Note that a Pfaffian equation associated with $M$ depends on the choice of 
the basis of $\mathbf C(x)\otimes_{\mathbf C[x]}M$, and it is not unique.

\section{Integral representation of the probability content of a polyhedron}
\label{sec3}
In this section, we show that the function $\varphi(a, b)$ in \eqref{p0}
can be regarded as a real analytic function.
Since the Heaviside function $H(x)$ is the hyperfunction defined by 
$-\frac{1}{2\pi\sqrt{-1}}\log(-z)$, 
we can expect the function $\varphi(a, b)$ to be expressed in terms of a 
logarithmic function.
However, we cannot find a suitable $d$-simplex for the $d$-form 
obtained by replacing $H\left(\sum_{i=1}^d a_{ij}x_i+b_j\right)$ 
in the integrand of \eqref{p0}
with 
$-\frac{1}{2\pi\sqrt{-1}}
\log\left(-\left(\sum_{i=1}^d a_{ij}x_i+b_j\right)\right)$.
In order to overcome this difficulty, 
we use a decomposition of $\varphi(a, b)$, 
which will be given in \eqref{eq:varphi_decomp}, 
and show that $\varphi(a, b)$ 
can be written as a linear combination of complex integrals.
This implies that $\varphi(a, b)$ is a real analytic function.

First, let us review some notions of polyhedra.
In the remainder of this paper, we will assume that $d$ and $n$ are positive integers.
A subset $H\subset \mathbf R^d$ is called a {\it half-space} 
if $H$ can be written as 
$
H=\{x\in \mathbf R^d | \sum_{i=1}^d a_ix_i +a_0 \geq 0\}
$
for some $a_i, a_0\in \mathbf R$.
A {\it polyhedron} is a finite intersection of half-spaces.
An inequality $\sum_{i=1}^d a_ix_i+b \geq 0$ is called {\it valid} 
for a polyhedron $P$
if all the points of $P$ satisfy the inequality.
We call a subset $S\subset \mathbf R^d$ an {\it affine subspace} 
if $S$ can be written as an intersection of hyperplanes. 
For a subset $S\subset \mathbf R^d$, 
the affine hull of $S$ is the smallest affine subspace 
which contains $S$, and we denote it by $\aff(S)$.
The {\it dimension} of a polyhedron $P$ is the dimension of $\aff(P)$,
and we denote it by $\dim(P)$.
For a polyhedron $P$ and 
an inequality $\sum_{i=1}^d a_ix_i+b \geq 0$ which is valid for $P$,
the intersection  
$$
F = P \cap \left\{x\in\mathbf R^d: \sum_{i=1}^d a_ix_i+b \geq 0 \right\}
$$
is called a {\it face} of $P$.
The dimension of a face $F$ is the dimension of $\aff(F)$.
A {\it facet} of a polyhedron $P$ is a face of $P$ whose dimension equals to 
$\dim(P)-1$.
For details, see \cite{ziegler}.

In order to describe the combinatorial structure of a polyhedron, 
we use the notion of the abstract simplicial complex \cite{edelsbrunner}.
Let $\mathcal F$ be a set consisting of subsets of 
$[n]:=\{1, \, 2, \dots, n\}$.
We call $\mathcal F$ an abstract simplicial complex when
$J\in \mathcal F$ and $J'\subset J$ implies $J'\in \mathcal F$.
Let $\mathcal F$ and $\mathcal F'$ be two abstract simplicial complexes.
We say that $\mathcal F$ is {\it equal} to $\mathcal F'$ 
and denote $\mathcal F=\mathcal F'$ 
when $\mathcal F$ and $\mathcal F'$ are equal as sets.
We say that $\mathcal F$ and $\mathcal F'$ are equivalent 
and denote $\mathcal F\cong\mathcal F'$ 
when there is a bijection $\sigma:\mathcal F\rightarrow \mathcal F'$ such that
$J_1\subset J_2$ if and only if $\sigma(J_1)\subset \sigma(J_2)$.

Let $P\subset \mathbf R^d$ be a polyhedron, and let 
$F_1, \dots, F_n$ be all the facets of $P$.
For each facet $F_j$, there is a unique half-space $H_j\subset \mathbf R^d$ 
that satisfies $(\partial H_j)\cap P = F_j$ and $P \subset H_j$
(see, e.g., exercise 2.14(iv) of Lecture 2 in \cite{ziegler}).
We call $\mathcal H= \{H_1, \dots, H_n\} $ 
{\it the family of the bounding half-spaces for the polyhedron $P$.} 
The {\it nerve} of $\{F_1, \dots, F_n\} $ is the abstract simplicial complex 
defined by 
\[
\mathcal F = \{J\subset [n] : F_J \neq \emptyset \}, 
\quad \left(F_J:= \bigcap_{j\in J} F_j\right).
\]
We also call $\mathcal F$
the {\it abstract simplicial complex associated with the polyhedron $P$}.
When $\mathcal F$ is an abstract simplicial complex associated 
with a polyhedron $P$, 
we have $\{j\} \in \mathcal F$ for any $j\in [n]$.

Next, we introduce the notion of a polyhedron in the ``general position.''
Since we need to consider information for points at infinity, 
we will use the idea of ``homogenization.''
The {\it homogenization} $\hat H$ of a half-space 
$H=\{x\in \mathbf R^d | \sum a_ix_i + a_0 \geq 0\} $
is defined as 
\[
\hat H = \{(x_0, \dots, x_d)\in \mathbf R^{d+1}  | \sum_{i=0}^d a_ix_i\geq 0\}.
\]
For a family of half-spaces $\mathcal H = \{H_1, \dots, H_n\} $, 
we call $\hat{\mathcal H}  =\{\hat H_0, \hat H_1, \dots, \hat H_n \} $
the {\it homogenization} of $\mathcal H$.
Here, we put $\hat H_0  = \{x_0 \geq 0\} $.
We say that a family of half-spaces $\mathcal H=\{H_1, \dots, H_n\} $ 
(or its homogenization $\hat{\mathcal H} =\{\hat H_0, \dots, \hat H_n \} $)
is {\it in general position} 
when, for $J\subset \{0, 1, \dots, n\}$, 
\[
\hat F_J:=
\left(\bigcap_{j\in J}  \partial\hat H_j\right)
\cap
\left(\bigcap_{j=0}^{n}  \hat H_j\right)
\]
is a $d+1-|J|$-dimensional cone 
(i.e., the affine hull of the cone is $d+1-|J|$-dimensional affine space)
or $\{0\} $.
This is somewhat analogous to the ``general position'' for hyperplane arrangements
in \cite[Chap 2. Section 9]{aomoto-kita}, 
but we emphasize that they are different 
(see Example \ref{ex:2}).
The polyhedron $P $ is {\it in general position}
when the family $\mathcal H$ of the bounding half-spaces of $P$ 
is in general position.

\begin{example}\label{ex:2}
Let $d=2$. We define  $H_j\, (1\leq j\leq 4)$ by 
\begin{align*}
H_1 &:= \left\{(x_1, x_2)\in \mathbf R^2 : x_1\geq 0 \right\}, &
H_2 &:= \left\{(x_1, x_2)\in \mathbf R^2 : x_1\leq 1 \right\}, \\
H_3 &:= \left\{(x_1, x_2)\in \mathbf R^2 : x_2\geq 0 \right\}, &
H_4 &:= \left\{(x_1, x_2)\in \mathbf R^2 : x_2\leq 1 \right\}.
\end{align*}
Then, the family of half-spaces $\mathcal H:= \{H_1, H_2, H_3, H_4\}$ is in general position.
However, the family of half-spaces $\mathcal H':= \{H_1, H_2, H_3\}$ is not in general position.
In fact, the homogenization of $\mathcal H$ and $\mathcal H'$ 
can be written as 
\begin{align*}
\hat{\mathcal H}
&= \left\{\hat H_0, \hat H_1, \hat H_2, \hat H_3, \hat H_4\right\}, &
\hat{\mathcal H}' 
&= \left\{\hat H_0, \hat H_1, \hat H_2, \hat H_3\right\}
\end{align*}
where 
\begin{align*}
\hat H_0 &:= \left\{(x_0, x_1, x_2)\in \mathbf R^{2+1} : x_0\geq 0 \right\}, \\
\hat H_1 &:= \left\{(x_0, x_1, x_2)\in \mathbf R^{2+1} : x_1\geq 0 \right\}, &
\hat H_2 &:= \left\{(x_0, x_1, x_2)\in \mathbf R^{2+1} : x_1\leq x_0 \right\}, \\
\hat H_3 &:= \left\{(x_0, x_1, x_2)\in \mathbf R^{2+1} : x_2\geq 0 \right\}, &
\hat H_4 &:= \left\{(x_0, x_1, x_2)\in \mathbf R^{2+1} : x_2\leq x_0 \right\}. 
\end{align*}
Calculating $\hat F_J$ for each $J\subset \{0, 1, 2, 3, 4\}$, 
we can show that $\hat{\mathcal H}$ is in general position.
For example, 
the set 
$
\hat H_0 \cap 
\partial\hat H_1 \cap \hat H_2 \cap \partial\hat H_3 \cap \hat H_4
$
is a $d+1-2$-dimensional cone, 
and the set 
$
\hat H_0 \cap 
\partial\hat H_1 \cap \partial\hat H_2 \cap \hat H_3 \cap \hat H_4
$
is equal to  $\{0\}$.
On the other hand, 
the family $\hat{\mathcal H}'$ is not in general position 
since the dimension of 
\[
\aff\left(
\partial\hat H_0 \cap 
\partial\hat H_1 \cap \partial\hat H_2 \cap \hat H_3
\right)
= \{x \in \mathbf R^{d+1}: x_0=x_1=0\}
\]
is not equal to $d+1-3=0$.

Hence, 
the polyhedron $\bigcap_{j=1}^4H_j$ in Figure \ref{fig:1}(a) 
is in general position, but 
the polyhedron $\bigcap_{j=1}^3H_j$ in Figure \ref{fig:1}(b) 
is not in general position.

We note that the hyperplane arrangement 
$\{\partial H_1, \, \partial H_2, \, \partial H_3, \, \partial H_4\}$
is not in general position in the sense of \cite{aomoto-kita}, 
but $\mathcal H$ is in general position.
\end{example}
\begin{example}
Let $H_5:= \left\{(x_1, x_2)\in \mathbf R^2 : x_1+x_2\geq 0 \right\}$.
Using the notation in Example \ref{ex:2}, 
the family of the half-space $\{H_j:1\leq j\leq 5\}$, 
which is shown in Figure \ref{fig:1}(c), 
is not in general position.
However, the polyhedron $\bigcap_{j=1}^5 H_j$ is in general position
since the family of the bounding half-spaces is $\{H_j:1\leq j\leq 4\}$.
\end{example}
\begin{figure}[htbp]
\center
\begin{tabular}{ccc}
\includegraphics[width=0.25\hsize]{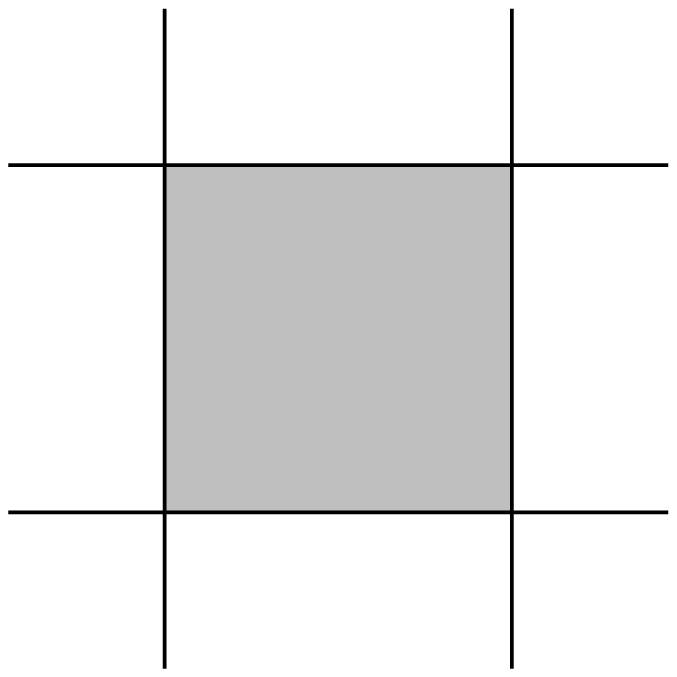}&
\includegraphics[width=0.25\hsize]{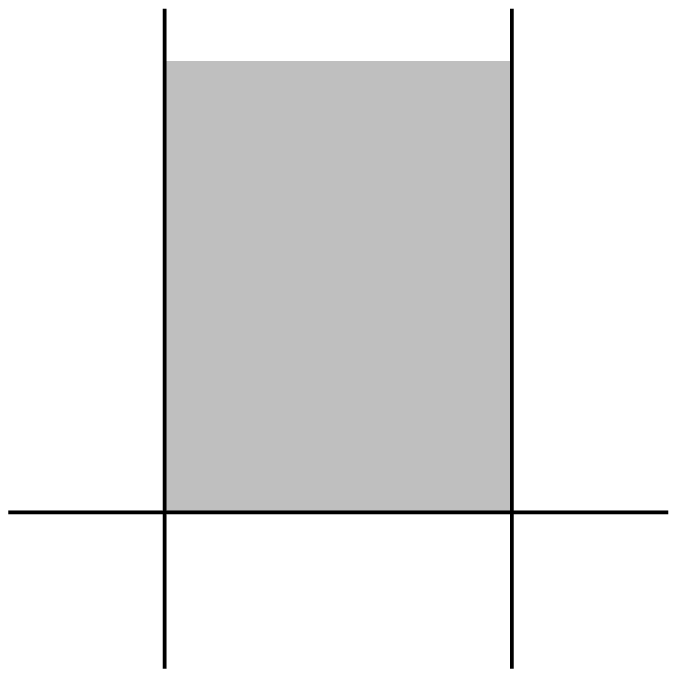}&
\includegraphics[width=0.25\hsize]{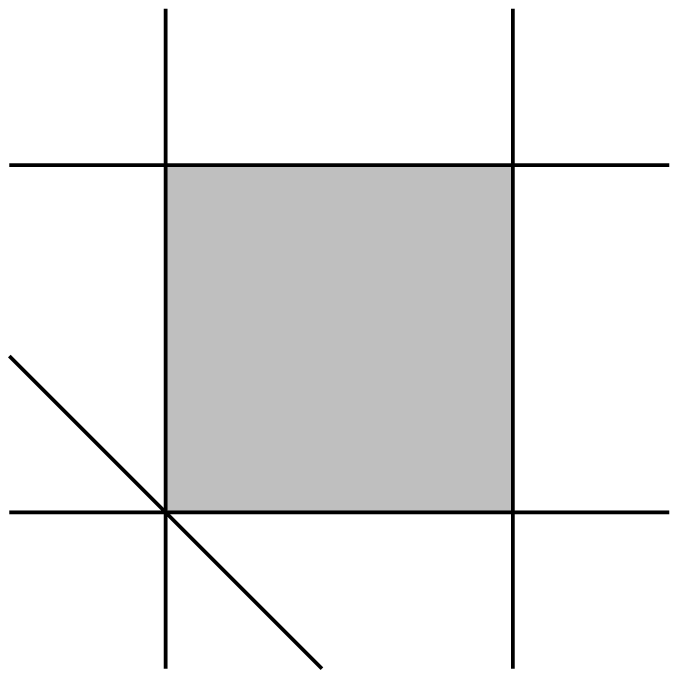}\\
(a) & (b) & (c) 
\end{tabular}
\caption{Examples of polyhedra}\label{fig:1}
\end{figure}
\begin{remark}
Note that our definition of general position is more restrictive than 
that of \cite{naiman-wynn}, 
and it is less restrictive than that of \cite{kuriki-miwa-hayter}.
For example, the polyhedron $\bigcap_{j=1}^3H_j$ in Example \ref{ex:2} 
is in general position by the definition in \cite{naiman-wynn};
and the polyhedron $\bigcap_{j=1}^4H_j$ in Example \ref{ex:2} 
is not in general position by the definition in \cite{kuriki-miwa-hayter}.
\end{remark}

Let $P\subset \mathbf R^d$ be a polyhedron.
Suppose the family of bounding half-spaces for $P$ is given by 
\[
\left\{
x \in \mathbf R^d: \sum_{i=1}^d\tilde a_{ij}x_i+\tilde b_j \geq 0
\right\}
\quad (1\leq j\leq n).
\]
We denote by $\tilde a$ the $d\times n$ matrix $(\tilde a_{ij})$, 
and by $\tilde b$ the vector $(\tilde b_1, \dots, \tilde b_n)$.
Let $F_j$ be the intersection of $P$ and 
the hyperplane $\{x\in\mathbf R^d:\sum_{i=1}^d\tilde a_{ij}x_i+\tilde b_j=0\}$.
The sets $F_1, \dots, F_n$ are all of the facets of $P$.
Let $\mathcal F$ be the nerve of $\{F_1, \dots, F_n\}$, 
which is the abstract simplicial complex of $P$.

Edelsbrunner showed the inclusion--exclusion identity 
for the indicator function of a polyhedron \cite[Lemma 5.1.]{edelsbrunner}.
\begin{proposition} [Edelsbrunner]
If $\mathcal F$ is the abstract simplicial complex associated with a  polyhedron $P$, 
then the indicator function of $P$ can be written as 
\[
\mathbf 1_P = \sum_{J\in\mathcal F} \prod_{j\in J} \left(\mathbf 1_{H_j} -1\right).
\]
\end{proposition}
\begin{example}
For the polyhedron $\bigcap_{j=1}^4H_j$ in Example \ref{ex:2}, 
the inclusion--exclusion identity can be written as follows:
\begin{align*}
\mathbf 1_{\bigcap_{j=1}^4H_j}
&= 1
+(\mathbf 1_{H_1}-1) +(\mathbf 1_{H_2}-1) +(\mathbf 1_{H_3}-1) +(\mathbf 1_{H_4}-1) 
\\
&\quad 
+(\mathbf 1_{H_1}-1)(\mathbf 1_{H_3}-1) 
+(\mathbf 1_{H_1}-1)(\mathbf 1_{H_4}-1) \\
&\quad
+(\mathbf 1_{H_2}-1)(\mathbf 1_{H_3}-1) 
+(\mathbf 1_{H_2}-1)(\mathbf 1_{H_4}-1).
\end{align*}
The first term of the right-hand side corresponds to the empty set.
\end{example}
With the Heaviside function $H$, 
the Edelsbrunner's identity can be written as 
\[
\prod_{j=1}^n H\left(\sum_{i=1}^d \tilde a_{ij}x_i+\tilde b_j\right)
 = \sum_{J\in\mathcal F} \prod_{j\in J} 
\left(H\left(\sum_{i=1}^d \tilde a_{ij}x_i+\tilde b_j\right) -1\right).
\]
Under the general position assumption, 
this identity can be generalized as follows.
\begin{theorem} \label{ptb:8}
In the notation above, 
if the polyhedron $P$ is  in general position, 
then there exists a neighborhood $U$ of $(\tilde a, \tilde b)$ 
such that the equation
\[
\prod_{j=1}^n H\left(\sum_{i=1}^d a_{ij}x_i+b_j\right)
 = \sum_{J\in\mathcal F} \prod_{j\in J} 
\left(H\left(\sum_{i=1}^d a_{ij}x_i+b_j\right) -1\right)
\]
holds for all $(a, b, x)\in U\times \mathbf R^d$.
\end{theorem}
Our proof of Theorem \ref{ptb:8} is technical, 
and it is independent from the other parts of this paper; thus, 
it is given in the appendix.

Consider $n$ polynomials 
$f_j(a, b, x) = \sum_{i=1}^d a_{ij}x_i +b_j \, (1\leq j\leq n)$
with variables $a_{ij}, \, b_j, \, x_i \, (1\leq i\leq d, 1\leq j\leq n)$, 
and let 
\begin{equation}\label{chiF}
\chi_F(a, b, x)= \prod_{j\in F} \left(H(f_j(a, b, x))-1\right)
\end{equation}
for $F\in \mathcal F$.
Note that $\chi_\emptyset(a, b, x)=1$.
We put 
\[
\varphi_F(a, b) 
=  \int_{\mathbf R^d} \frac{1}{(2\pi)^{d/2}}
\exp(-\frac{1}{2}\sum_{i=1}^dx_i^2) \chi_F(a, b, x) dx
\quad ( F\in \mathcal F).
\]
By Theorem \ref{ptb:8}, 
the function $\varphi(a, b)$ in \eqref{p0} can be decomposed as 
\begin{equation}\label{eq:varphi_decomp}
\varphi(a, b) = \sum_{F\in {\cal F}} \varphi_F(a, b)
\end{equation}
on a neighborhood of $(\tilde a, \tilde b)$
if the polyhedron $P$ is in general position.

In order to give analytic continuations of 
the function $\varphi(a, b)$, 
it is enough to consider $\varphi_F(a, b)$.
For $F\in\mathcal F$, 
let $\alpha_F(a)$ be an $|F|\times |F|$ matrix, 
where $|F|$ is the number of the elements in $F$ and 
\begin{equation}\label{alphaF}
\alpha_F(a)=(\alpha_{ij}(a))_{i, j\in F}, \quad
\alpha_{ij}(a) = \sum_{k=1}^d a_{ki}a_{kj} \quad (1\leq i, j\leq n).
\end{equation}
This is a submatrix of the Gram matrix of $a$.
The matrices $\alpha_F(a)$ are symmetric and positive semidefinite.
Since the function $\varphi_F(a, b)$ can be written as 
\[
\int_{\mathbf R^d} \frac{1}{(2\pi)^{d/2}}
(-1)^{|F|}
\exp(-\frac{1}{2}\sum_{i=1}^dx_i^2) \prod_{j\in F} H(-f_j(a, b, x)), 
\]
we can expect that $\varphi_F(a, b)$ is written by a $d$-simplex and 
the $d$-form
\[
\frac{1}{(2\pi)^{d/2}} \exp(-\frac{1}{2}\sum_{i=1}^dz_i^2)
\prod_{j\in F} \left(\frac{\log(f_j(a, b, z))}{2\pi\sqrt{-1}}\right) dz.
\]
In fact, we can find a suitable $d$-simplex and thus have 
Proposition \ref{prop:c-int}.

\begin{lemma}\label{lem:c-int}
If the polyhedron $P$ is in general position, 
then for $F\in \mathcal F$, 
the value of $\varphi_F(\tilde a, \tilde b )$ can be written as 
\begin{equation}\label{int1}
\sum_{\lambda \in \{\pm 1\}^{|F|}} \int_{\gamma^\lambda}
\frac{1}{(2\pi)^{d/2}} \exp(-\frac{1}{2}\sum_{i=1}^dz_i^2)
\prod_{j\in F} \left(\frac{\log(f_j(\tilde a, \tilde b, z))}{2\pi\sqrt{-1}}\right) dz.
\end{equation}
Here, for $\lambda \in \{\pm 1\}^{|F|}$, 
$\gamma^\lambda$ is a smooth map from $\mathbf R^d$ to $\mathbf C^d$.
We suppose the multivalued function $\log$  satisfies $\log(1)=0$
and the branch cut is $\{z\in \mathbf C:\Re(z) \leq 0\}$. 
\end{lemma}
\begin{proof}
Let $s$ be the number of elements in $F$.
By the general position assumption, $s$ is not greater than $d$.
We denote by $j(1), \dots, j(s)$ all of the elements of $F$.
Since the polyhedron $P$ is in general position, 
the vectors $u_k=(\tilde a_{1j(k)}, \dots, \tilde a_{dj(k)})^\top(1\leq k \leq s)$
 are linearly independent
\comment{(see Corollary \ref{ptb:13} in Section \ref{sec:6})}.
Consequently, the determinant $\alpha_F(\tilde a)$ is not zero.
Let $u_k (s < k\leq d)$ be an orthonormal basis of 
the orthogonal complement of the subspace
$\sum_{k=1}^s \mathbf Ru_k \subset\mathbf R^d$.
We denote the vector $u_k$ by $u_k=(u_{1k}, \dots, u_{dk})^\top$.
The matrix $u=(u_{ij})$ is regular, and we set $u^{-1}=(u^{ij})$.
Without loss of generality, we can assume $\det u > 0$.
Under this assumption, we have $\det u = \sqrt{|\det\alpha_F(\tilde a)|}$.
Let $\varepsilon(t)$ be a positive bounded function on $\mathbf R$, i.e., 
$\inf_{t\in\mathbf R} \varepsilon(t) > 0$ 
and $\sup_{t\in \mathbf R}\varepsilon(t) < \infty$.
For a vector $(\lambda_1, \dots, \lambda_s) \in \{-1, 1\}^s$, 
we define $\gamma^\lambda:\mathbf R^d\rightarrow \mathbf C^d$ by
\[
\gamma^\lambda_j(t)
= \sum_{i=1}^d u^{ij} 
  \left(\lambda_it_i + \lambda_i\sqrt{-1}\varepsilon(\lambda_it_i)\right)
  \quad (t\in \mathbf R^d).
\]
Here, we put $\lambda_j = 1$ for $s<j\leq d$.

By the coordinate transformation $w_j=\sum_{i=1}^du_{ij}z_i$, 
the integral \eqref{int1} can be written as 
\[
\frac{1}{(2\pi)^{d/2}} 
\frac{1}{|\alpha_F(\tilde a )|^{1/2}}
\sum_{\lambda \in \{\pm 1\}^s} \int_{\gamma'_\lambda}
\exp(-\frac{1}{2}\sum_{i=1}^d (\sum_{j=1}^d u^{ji}w_j)^2) 
\prod_{k=1}^s \left(\frac{\log(w_k+\tilde b_k)}{2\pi\sqrt{-1}}\right) dw,
\]
where 
$
\gamma'_{\lambda j}(t) = 
\lambda_jt_j + \lambda_j\sqrt{-1}\varepsilon(\lambda_jt_j)\, (t\in \mathbf R^d).
$
Calculating this integral recursively, 
we have
\[
\frac{1}{(2\pi)^{d/2}} \frac{1}{|\alpha_F(\tilde a )|^{1/2}}
\int_E \exp(-\frac{1}{2}\sum_{i=1}^d (\sum_{j=1}^d u^{ji}y_j)^2) dy,
\]
where
$E = \{y\in \mathbf R^d: y_j + \tilde b_j\leq 0, \, j\in F\}$.
By the coordinate transformation $x_i=\sum_{j=1}^d u^{ji}y_j$, the above integral
is 
\[
\frac{1}{(2\pi)^{d/2}}
\int_{f_j(\tilde a, \tilde b, x)\leq 0, j\in F} \exp(-\frac{1}{2}\sum_{i=1}^d x_i^2) dx, 
\]
which is equal to $\varphi_F(\tilde a, \tilde b )$.
\end{proof}
Moreover, we have the following proposition.
\begin{proposition}\label{prop:c-int}
Let $U_F$ be a domain $\{(a, b): \det \alpha_F(a) \neq 0\}$.
The function $\varphi_F(a, b)$ can be written as 
\begin{equation}\label{varphi_F}
\sum_{\lambda \in \{\pm 1\}^{|F|}} \int_{\gamma_{\lambda}}
\frac{1}{(2\pi)^{d/2}} \exp(-\frac{1}{2}\sum_{i=1}^dz_i^2)
\prod_{j\in F} \left(-\frac{\log(f_j(a, b, z))}{2\pi\sqrt{-1}}\right) dz
\end{equation}
on a connected open neighborhood of $(\tilde a, \tilde b )$ in $U_F$.
Here, $\gamma_{\lambda}$ denotes the integral path $\gamma^\lambda$ in 
Lemma \ref{lem:c-int}.
And we suppose the multivalued function $\log$  satisfies $\log(1)=0$
and the branch cut is $\{z\in \mathbf C:\Re(z) \leq 0\}$. 
\end{proposition}
\begin{proof}
Since $\det(\alpha_F(a))\neq 0$, 
by arguments similar to those in the proof of Lemma \ref{lem:c-int}, 
$\varphi_F(a, b)$ can be written as 
\[
\sum_{\lambda \in \{\pm 1\}^{|F|}} \int_{\gamma'_\lambda}
\frac{1}{(2\pi)^{d/2}} \exp(-\frac{1}{2}\sum_{i=1}^dz_i^2)
\prod_{j\in F} \left(\frac{\log(f_j(a, b, z))}{2\pi\sqrt{-1}}\right) dz.
\]
The matrix $u'$ and the integral path $\gamma'_\lambda$
can be constructed similarly.

We need to show that the above integral is equal to \eqref{varphi_F}.
There is a smooth path $u(t)\, (t\in[0, 1])$ in the general linear group of degree $d$ over $\mathbf C$
such that $u(0)=u$ and $u(1)=u'$. 
The homotopy between $\gamma_\lambda$ and $\gamma'_\lambda$ is
given by 
\[
\gamma_{\lambda j}(s)(t)
= \sum_{k=1}^d u^{jk}(s) \left(\lambda_kt_k + \lambda_k\sqrt{-1}\varepsilon(\lambda_kt_k)\right) \quad (t\in \mathbf R^d).
\]
Consequently, the value of the integral on the right-hand side of \eqref{varphi_F} does not change
when we change the integral path with $\gamma'_\lambda$.
\end{proof}
\begin{corollary}
The function $\varphi(a, b)$ is a real analytic function, 
and it has an analytic continuation along 
every path in $\bigcap_{F\in{\cal F}} U_F$.
\end{corollary}

\section{Holonomic modules}
\label{sec4}
In this section, 
we explicitly give a system of differential equations for the function 
\begin{equation}\label{fn:chiP}
\chi_P (a, b, x)
= \sum_{F\in \mathcal F} \chi_F(a, b, x), 
\end{equation}
and show that the system is holonomic.
Note that the function $\chi_F$ is defined by \eqref{chiF}.
Since the equation
\[
\varphi(a, b) =
\int_{\mathbf R^d} \frac{1}{(2\pi)^{d/2}}
\exp(-\frac{1}{2}\sum_{i=1}^dx_i^2) \chi_P(a, b, x) dx
\]
holds on a neighborhood of $(\tilde a, \tilde b)$,
a holonomic system for $\varphi(a, b)$ can be given 
as the integration module of a holonomic module
associated with 
\[
\exp(-\frac{1}{2}\sum_{i=1}^dx_i^2) \chi_P(a, b, x).
\]
For more about the integration module, 
see \cite{bjork} and \cite{Oaku1997}.

For $J\subset [n]$, we define a hyperfunction $\chi^J$ by 
\[
\chi^J = \pd{b}^J\chi_P \quad (\pd{b}^J := \prod_{i\in J}\pd{b_i} ).
\]
Note that $\chi^\emptyset=\chi_P$.
\begin{lemma}\label{lem:2}
If $J \subset [n]$ is not an element of ${\cal F}$, 
then we have $\pd{b}^J\chi_F = 0$ for $F\in \mathcal F$.
Consequently, we have $\chi^J = 0$.
\end{lemma}
\begin{proof}
Since $\mathcal F$ is an abstract simplicial complex, 
we have $J\not\subset F$.
Take $k\in J\backslash F$, then we have $\pd{b}^J\chi_F=0$,
since $\pd{b_k} \prod_{j\in F} \left(H(f_j(a, b, x))-1\right) = 0$.
\end{proof}
We now provide a system of differential equations for the $\chi^J$'s.
Let $g=(g^J)_{J\in \mathcal F}$ be a vector whose elements are functions indexed by the set $\mathcal F$.
Let us consider the system 
defined by the following:
\begin{align}
\label{eq:chi0} \pd{x_i}g^J&=\sum_{j=1}^n a_{ij}\pd{b_j}g^J 
&&\hspace{-36pt}(1\leq i\leq d, \, J\in \mathcal F), \\
\label{eq:chi1} \pd{a_{ij}}g^J &=x_i\pd{b_j} g^J 
&&\hspace{-36pt}(1\leq i\leq d, \, 1\leq j\leq n, \, J\in \mathcal F), \\
\label{eq:chi2} \pd{b_j} g^J &=g^{J\cup \{j\}} 
&&\hspace{-36pt}(j\in J^c, \, J\in \mathcal F), \\
\label{eq:chi3} f_jg^J&=0 
&&\hspace{-36pt}(j\in  J, \, J\in {\cal F}), 
\end{align}
where $g^{J\cup \{j\}}=0$ for $J\cup \{j\} \notin {\cal F}$.

\begin{lemma}\label{lem:1}
Let $g=(\pd{b}^J\chi_F)_{J\in \mathcal F}$ for $F\in \mathcal F$.
Then the function $g$ satisfies equations
\eqref{eq:chi0}, \eqref{eq:chi1}, \eqref{eq:chi2}, and \eqref{eq:chi3}.
\end{lemma}
\begin{proof}
When $F=\emptyset$, it is obvious that equations 
\eqref{eq:chi0}, \eqref{eq:chi1}, \eqref{eq:chi2}, and \eqref{eq:chi3}
hold, since $\chi_\emptyset = 1$.
Suppose $F$ is not the empty set.

We first check equation \eqref{eq:chi0}.
Since  
$
\pd{b_j}\chi_F(a, b, x) =0
$
for $j\notin F$, 
we have 
\[
\pd{x_i}\chi_F(a, b, x)
= \sum_{j\in F} a_{ij}\pd{b_j}\chi_F(a, b, x)
= \sum_{j=1}^n a_{ij}\pd{b_j}\chi_F(a, b, x).
\]
Here, we apply the chain rule for hyperfunctions.
The equation above implies \eqref{eq:chi0}.

We now show equation \eqref{eq:chi1}.
When $j\notin F$, both sides of \eqref{eq:chi1} are equal to $0$.
When $j\in F$, we have $\pd{a_{ij}}g^J=x_i\pd{b_j} g^J$
by the chain rule.

Equation \eqref{eq:chi2} holds by Lemma \ref{lem:2}.
For $J\subset F$, we have
\[
g^J=
\prod_{k\in J}\delta(f_k(a, b, x))
\prod_{k\in F\backslash J} \left(H(f_k(a, b, x))-1\right).
\]
This implies \eqref{eq:chi3}.
\end{proof}

By \eqref{fn:chiP}, we have 
$
\chi^J = \sum_{F\in \mathcal F} \pd{b}^J\chi_F.
$
By this equation and Lemma \ref{lem:1}, we have the following proposition.
\begin{proposition}\label{ann_chiP}
The vector-valued function $g=(\chi^J)_{J\in \mathcal F}$ satisfies 
equations \eqref{eq:chi0}, \eqref{eq:chi1}, \eqref{eq:chi2}, 
and \eqref{eq:chi3}.
\end{proposition}
Next, we show that the system defined by 
\eqref{eq:chi0}, \eqref{eq:chi1}, \eqref{eq:chi2}, and \eqref{eq:chi3} 
for $\chi_P$ is a holonomic system.
In the remainder of this paper, we will frequently use the following rings:
\begin{align*}
D_{abx} &:=\mathbf C\langle a_{ij}, b_j, x_i, \pd{a_{ij}}, \pd{b_j}, \pd{x_i}: 1\leq i\leq d, \, 1\leq j\leq n \rangle, \\
D_{ab}  &:=\mathbf C\langle a_{ij}, b_j,    \pd{a_{ij}}, \pd{b_j}: 1\leq i\leq d, \, 1\leq j\leq n \rangle, \\
\mathbf C[a, b, x, \xi_a, \xi_b, \xi_x]&:= \mathbf C[a_{ij}, b_j, x_i, \xi_{a_{ij}}, \xi_{b_j}, \xi_{x_i}:\, 1\leq i\leq d, \, 1\leq j\leq n].\\
\end{align*}
We also use the free modules $(D_{abx})^{|\mathcal F|}, (D_{ab})^{|\mathcal F|}$, and $\mathbf C[a, b, x]^{|\mathcal F|}$, 
whose basis is $\{g^J:J\in \mathcal F\}$.
\begin{proposition}\label{prop:chiP}
Let $N_\chi$ be the sub left $D_{abx}$-module of $(D_{abx})^{|\mathcal F|}$ generated by 
\begin{align}
\label{op:chi0} \left(\pd{x_i}-\sum_{j=1}^n a_{ij}\pd{b_j}\right)g^J
&\quad (1\leq i\leq d, \, J\in \mathcal F), \\
\label{op:chi1} \left(\pd{a_{ij}}-x_i\pd{b_j}\right) g^J
&\quad (1\leq i\leq d, \, 1\leq j\leq n, \, J\in \mathcal F), \\
\label{op:chi2} \pd{b_j} g^J - g^{J\cup \{j\}}
&\quad (j\in J^c, \, J\in \mathcal F), \\
\label{op:chi3}  f_jg^J, \, &\quad (j\in J, J\in\mathcal F).
\end{align}
Then, the quotient module $M_\chi=(D_{abx})^{|\mathcal F|}/N_\chi$ is holonomic.
\end{proposition}
\begin{proof}
The principal symbols of 
\eqref{op:chi0}, \eqref{op:chi1}, \eqref{op:chi2}, and \eqref{op:chi3}
are the following:
\begin{align}
\label{char0} (\xi_{x_i}-\sum_{j=1}^n a_{ij}\xi_{b_j})g^J&\quad  (1\leq i\leq d, \, J\in \mathcal F), \\
\label{char1} (\xi_{a_{ij}}-x_i\xi_{b_j})g^J  &\quad (1\leq i\leq d, \, 1\leq j\leq n, \, J\in \mathcal F), \\
\label{char2} f_jg^J, \, \xi_{b_k} g^J &\quad (j\in  J, \, k\in J^c, \, J\in \mathcal F).
\end{align}
For $J\in \mathcal F$, let $V_J$ be an algebraic variety defined by
\begin{align}
\label{V_J0} \xi_{x_i}-\sum_{j=1}^n a_{ij}\xi_{b_j}, \, \xi_{a_{ij}} - x_i\xi_{b_j} &
 \quad (1\leq i\leq d, \, 1\leq j\leq n), \\
\label{V_J1} f_j, \, \xi_{b_k} &
 \quad (j\in  J, \, k\in J^c).
\end{align}
By  \cite[Proposition 1]{oaku1994}, 
the union $\bigcup_{J\in \mathcal F} V_J$ includes $\ch{M}$.
Since the rank of the Jacobian matrix of \eqref{V_J0} and \eqref{V_J1} 
is $nd+n+d$,  
the Krull dimension of $V_J$ is equal to $nd+n+d$.
Hence, the dimension of $\ch{M}$ is not greater than $nd+n+d$.
\end{proof}

Let $N_q$ be a $D_{abx}$-submodule of $(D_{abx})^{|\mathcal F|}$ generated by 
\eqref{op:chi1}, \eqref{op:chi2}, \eqref{op:chi3}, and
\begin{equation}\label{op:q0}
\left(x_i+\pd{x_i}-\sum_{j=1}^n a_{ij}\pd{b_j}\right)g^J \quad (1\leq i\leq d, \, J\in \mathcal F).
\end{equation}
Proposition \ref{prop:chiP} implies 
that the quotient module $M_q=(D_{abx})^{|\mathcal F|)}/N_q$
is a holonomic module.
Moreover, the function
\[
q^J(a, b, x) = \exp(-\frac{1}{2}\sum_{i=1}^dx_i^2)\chi^J  \quad (J\in \mathcal F)
\]
is a solution of $N_q$ (see \cite{ost}).
By Lemma \ref{lem:1}, the function 
\[
\exp(-\frac{1}{2}\sum_{i=1}^dx_i^2)\pd{b}^J\chi_F  \quad (J\in \mathcal F)
\]
is also a solution of $N_q$ for $F\in \mathcal F$.

Calculating the integration module of $N_q$ with respect to $x$, 
we have the following theorem.
\begin{theorem}\label{th:hol}
Let $N$ be the sub left $D_{ab}$-module of $(D_{ab})^{|\mathcal F|}$ generated by 
\begin{align}
\label{op:p1} \left(\pd{a_{ij}}-\sum_{k=1}^n a_{ik}\pd{b_k}\pd{b_j}\right) g^J 
&\quad (1\leq i\leq d, \, 1\leq j\leq n, \, J\in \mathcal F), \\
\label{op:p2}  \pd{b_j} g^J - g^{J\cup \{j\}} &\quad (j\in J^c, \, J\in \mathcal F), \\
\label{op:p3}  (b_j+\sum_{k=1}^n\sum_{i=1}^da_{ij}a_{ik}\pd{b_k})g^J &\quad (j\in J, \, J\in {\cal F}).
\end{align}
Then the quotient module $M=(D_{ab})^{|\mathcal F|}/N$ is isomorphic to 
the integration module $M_q/\sum_{i=1}\pd{x_i}M_q$.
Consequently, $M$ is a holonomic module.
\end{theorem}
\begin{proof}
We denote by $\iota$
the canonical morphism from $(D_{ab})^{|\mathcal F|}$ to 
the integration module $M_q/\sum_{i=1}\pd{x_i}M_q$.
Let $P_i=x_i+\pd{x_i}-\sum_{k=1}^n a_{ik}\pd{b_k}$.
Since we have
\begin{align}
\label{eq1}
\left(\pd{a_{ij}}-x_i\pd{b_j}\right) g^J + \pd{b_j}P_ig^J 
&=\left(\pd{x_i}\pd{b_j}+\pd{a_{ij}}-\sum_{k=1}^n a_{ik}\pd{b_k}\pd{b_j}\right) g^J, \\
\label{eq2}
f_jg^J-\sum_{i=1}^da_{ij}P_ig^J 
&=\left(b_j-\sum_{i=1}^da_{ij}\pd{x_i}
+\sum_{k=1}^n\sum_{i=1}^da_{ij}a_{ik}\pd{b_k}\right)g^J, 
\end{align}
the $D_{abx}$-module $N_q$ is generated by \eqref{op:q0}, \eqref{op:chi2}, and
\begin{align*}
& \left(\pd{x_i}\pd{b_j}+\pd{a_{ij}}-\sum_{k=1}^na_{ik}\pd{b_k}\pd{b_j}\right)g^J
\quad (1\leq i\leq d, \, J\in \mathcal F), \\
& \left(b_j-\sum_{i=1}^da_{ij}\pd{x_i}
+\sum_{k=1}^n\sum_{i=1}^da_{ij}a_{ik}\pd{b_k}\right)g^J
\quad (j\in  J, \, J\in {\cal F}).
\end{align*}
Consequently, we have $N\subset \ker\iota$.

Next, we show the opposed inclusion $\ker\iota\subset N$.
Regarding $\left(D_{ab}\right)^{|\mathcal F|}$
as a subset of $\left(D_{abx}\right)^{|\mathcal F|}$, 
the left $D_{ab}$-module $N_q +\sum \pd{x_i} (D_{abx})^{|\mathcal F|}$ 
includes $\ker\iota$.
Since the module $N_q$ is generated by 
\eqref{op:chi1}, \eqref{op:chi2}, \eqref{op:chi3}, and \eqref{op:q0}, 
the module $N_q +\sum \pd{x_i} (D_{abx})^{|\mathcal F|}$ can be written as 
\begin{equation}\label{module1}
\sum_{\lambda\in\Lambda} D_{abx}Q_\lambda 
+\sum_{J\in\mathcal F}\sum_{i=1}^d D_{abx}P_ig^J 
+\sum_{i=1}^d  \pd{x_i} (D_{abx})^{|\mathcal F|}.
\end{equation}
Here, we denote by $Q_\lambda\, (\lambda\in\Lambda)$ 
the differential operators in 
\eqref{op:chi1}, \eqref{op:chi2}, and \eqref{op:chi3}.
The left $D_{ab}$-module \eqref{module1} is equal to 
\begin{equation}\label{module2}
\sum_{\lambda\in\Lambda} D_{ab}Q_\lambda 
+\sum_{J\in\mathcal F}\sum_{i=1}^d D_{abx}P_ig^J 
+\sum_{i=1}^d  \pd{x_i} (D_{abx})^{|\mathcal F|}.
\end{equation}
Note that the first term of \eqref{module2} is different from 
that of \eqref{module1}.
In fact, 
for any $P_i$ and the differential operator in 
\eqref{op:chi1}, \eqref{op:chi2}, and \eqref{op:chi3}, 
we have 
\begin{align*}
P_i\left(\pd{a_{k\ell}}-x_k\pd{b_\ell}\right)g^J
&=\left(\pd{a_{k\ell}}-x_k\pd{b_\ell}\right)P_ig^J, \\
P_i\left(\pd{b_k}g^J-g^{J\cup \{k\}}\right)
&= \pd{b_k}P_ig^J -P_ig^{J\cup \{k\}}, \\
P_i\left(\sum_{k=1}^da_{k\ell}x_k+b_\ell\right)g^J
&=\left(\sum_{k=1}^da_{k\ell}x_k+b_\ell\right)P_ig^J.
\end{align*}
These equations imply that 
for any differential operator $Q_\lambda$ in 
\eqref{op:chi1}, \eqref{op:chi2}, and \eqref{op:chi3}, 
the operator $P_iQ_\lambda$ is an element of 
$\sum_{i, J} D_{abx}P_ig^J$.
Consequently, module \eqref{module2} includes
$
x_iQ_\lambda = 
\sum_{k=1}^n a_{ik}\pd{b_k}Q_\lambda +P_iQ_\lambda  -\pd{x_i}Q_\lambda.
$
By induction on the multi-index $\alpha\in {\bf N}_0^d$, 
module \eqref{module2} includes $x^\alpha Q_\lambda$ for any 
$\alpha\in {\bf N}_0^d$.
Hence, module \eqref{module2} includes \eqref{module1}.
The opposite inclusion is obvious.

We denote by $Q'_\lambda\, (\lambda\in\Lambda)$ 
the differential operators in \eqref{op:p1}, \eqref{op:p2}, and \eqref{op:p3}.
By \eqref{eq1} and \eqref{eq2}, the left $D_{ab}$-module
\[
\sum_{\lambda\in\Lambda} D_{ab}Q'_\lambda 
+\sum_{J\in\mathcal F}\sum_{i=1}^d D_{abx}P_ig^J 
+\sum_{i=1}^d  \pd{x_i} (D_{abx})^{|\mathcal F|}
\]
is equal to module \eqref{module2}.
Obviously, this module is equal to the left $D_{ab}$-module
\begin{equation}\label{module3}
\sum_{\lambda\in\Lambda} D_{ab}Q'_\lambda 
+\sum_{\alpha\in \mathbf N_0^d}\sum_{J\in\mathcal F}\sum_{i=1}^d x^\alpha D_{ab}P'_ig^J 
+\sum_{i=1}^d  \pd{x_i} (D_{abx})^{|\mathcal F|},
\end{equation}
where $P'_i:= x_i-\sum_{k=1}^n a_{ik}\pd{b_k}$.
Note that the module $N_q +\sum \pd{x_i} (D_{abx})^{|\mathcal F|}$ 
is equal to \eqref{module3}.
Since $\ker\iota$ is a subset of the intersection of 
$(D_{ab})^{|\mathcal F|}$ and \eqref{module3}, 
we have
\[
\ker\iota \subset
\sum_{\lambda\in\Lambda} D_{ab}Q'_\lambda 
+\sum_{\alpha\in \mathbf N_0^d}\sum_{J\in\mathcal F}\sum_{i=1}^d x^\alpha D_{ab}P'_ig^J.
\]

Let $P\in \ker\iota$. 
Then the element $P$ can be written as $Q+R$, where 
$Q\in \sum_{\lambda\in\Lambda} D_{ab}Q'_\lambda$
and 
$R \in \sum_{\alpha\in \mathbf N_0^d}\sum_{J\in\mathcal F}\sum_{i=1}^d x^\alpha D_{ab}P'_ig^J$.
The element $P-Q=R$ is an element of 
the $D_{abx}$-module $L:=\sum_{J\in\mathcal F}\sum_{i=1}^d D_{abx}P'_ig^J$.
Let $\prec$ be a lex order which satisfies $x_i \succ m$ 
for $1\leq i\leq d$ and any monomial $m\in D_{ab}$.
Since the Gr\"obner basis for $L$ with respect to $\prec$
is given by 
$
\left\{ P'_ig^J : 1\leq i\leq d, \, J\in\mathcal F \right\}, 
$
the leading term of $P-Q$ must be divided by some 
$x_ig^J\, (1\leq i \leq d, \, J\in\mathcal F)$.
Since $P-Q\in (D_{ab})^{|\mathcal F|}$, we have $P-Q=0$.
Therefore, we have 
$
\ker\iota \subset \sum D_{ab}Q'_\lambda.
$
\end{proof}

\section{Pfaffian equation and holonomic rank}
\label{sec5}
Let $M$ be the module derived in Theorem \ref{th:hol}.
In this section, we will evaluate the holonomic rank of $M$ and derive 
a Pfaffian equation associated with $M$.
The following lemma gives a lower bound of the holonomic rank.
\begin{lemma}\label{lem:liearly-independent}
The real analytic functions $\varphi_F(a, b)$,
where $F$ runs over the abstract simplicial complex $\mathcal F$ 
associated with $P$,
are linearly independent solutions of $M$.
Consequently, the holonomic rank of $M$ is not less than 
the number of the nonempty faces of $P$.
\end{lemma}
\begin{proof}
By Proposition \ref{prop:c-int} and Theorem \ref{th:hol}, it is obvious that 
the function $\varphi_F$ is a solution of $M$.

Suppose a matrix $\tilde a =(\tilde a_{ij})$ and 
a vector $\tilde b =(\tilde b_j)$ satisfy equation \eqref{calP}.
Note that $\tilde a_{ij}$ and $\tilde b_j$ are real numbers.
Let $U\subset \mathbf C^{d\times n}$ be a sufficiently small neighborhood of $\tilde a $.
Note that the function $\varphi_F$ is defined on $\{(a, b): a\in U, b\in \mathbf C^n\}$,
by Proposition \ref{prop:c-int}.
We prove the linear independence of these functions.
Suppose $\mathcal F=\{F_1, \dots, F_s\}$ and 
\[
|F_1| \leq |F_2| \leq \cdots\leq |F_s|.
\]
We denote $\varphi_{F_j}$ by $\varphi_j$.
Suppose $\sum_{i=1}^s c_j\varphi_j=0$ for some complex numbers 
$c_j\, (1\leq i\leq s)$.
Take an arbitrary $k$ such that $1\leq k \leq s$,
and suppose $c_1=\cdots=c_{k-1}=0$. It is enough to show that $c_k=0$.

Note that $F_\ell\not\subset F_k$ for $\ell>k$,
since $|F_\ell| \geq |F_k|$.
Define $a(t)$ and $b(t)$ as 
\[
a_{ij}(t) = \tilde a_{ij}, \quad
b_j(t)    = \begin{cases} \tilde b_j+t & j \in F_k\\ \tilde b_j-t & j \notin F_k \end{cases},
\quad t\in [0, \infty).
\]
Since there is an element of $F_\ell$ 
which is not included in $F_k$ for $\ell > k$, 
we have $\lim_{t\rightarrow \infty}\prod_{j\in F_\ell}H(f_j(a(t), b(t), x))=0$
for all $x\in \mathbf R^d$.
By the Lebesgue convergence theorem, we have 
\[
\lim_{t\rightarrow \infty}\varphi_\ell(a(t), b(t))= 0 \quad (\ell> k).
\]
Since we have $\lim_{t\rightarrow\infty}\prod_{j\in F_k}H(f_j(a(t), b(t), x))=1$ 
for all $x\in \mathbf R^d$, again
by the Lebesgue convergence theorem, we have
$\lim_{t\rightarrow \infty} \varphi_\ell(a(t), b(t))= 1$.
By the assumption of the induction, we have 
$
\sum_{j=k}^s c_j\varphi_j(a(t), b(t))=0.
$
Taking the limit of both sides as $t\rightarrow \infty$, 
we have $c_k= 0$.
\end{proof}

\begin{theorem}
The holonomic rank of $M$ is equal to the number of nonempty faces of $P$, 
i.e., 
\[
\mathrm{rank}(M) = |\mathcal F|.
\]
In addition, a Pfaffian equation associated with $M$ is given by 
\begin{align}
\label{pfaff1} \pd{a_{ij}}g^J &= \sum_{k=1}^na_{ik}\pd{b_k}\pd{b_j}g^J
\quad (1\leq i\leq d, \, 1\leq j\leq n, \, J\in \mathcal F), \\
\label{pfaff2} \pd{b_j} g^J &= g^{J\cup \{j\}} \quad (j\in J^c, \, J\in \mathcal F), \\
\label{pfaff3} \pd{b_j}g^J 
&= -\sum_{k\in J}\alpha^{jk}_J(a)\left(b_kg^J+\sum_{\ell\in J^c}\alpha_{k\ell}(a)g^{J\cup \{\ell\}}\right)
\quad (j\in J, \, J\in {\cal F}).
\end{align}
Here, $(\alpha^{ij}_F(a))_{i, j\in F}$ is the inverse matrix of $\alpha_F(a)$ 
in \eqref{alphaF}.
Note that the right-hand side of \eqref{pfaff1} can be reduced by 
\eqref{pfaff2} and \eqref{pfaff3}.
\end{theorem}
\begin{proof}
We have \eqref{pfaff1} and \eqref{pfaff2} 
by \eqref{op:p1} and \eqref{op:p2}, respectively.
By \eqref{op:p2} and \eqref{op:p3}, 
the right-hand side of \eqref{pfaff3} can be written as
\begin{align*}
-\sum_{k\in J}\alpha^{jk}_J\left(b_kg^J+\sum_{\ell\in J^c}\alpha_{k\ell}g^{J\cup \{\ell\}}\right)
&= \sum_{k\in J}\alpha^{jk}_J\left(\sum_{\ell=1}^n\alpha_{k\ell}\pd{b_\ell}g^J 
-\sum_{\ell\in J^c}\alpha_{k\ell}\pd{b_\ell}g^J\right)\\
&= \sum_{k\in J}\alpha^{jk}_J\sum_{\ell\in J}\alpha_{k\ell}\pd{b_\ell}g^J
=\pd{b_j}g^J.
\end{align*}
Consequently, we have \eqref{pfaff3}.

By \eqref{pfaff1}, \eqref{pfaff2}, and \eqref{pfaff3}, 
the module $\mathbf C(a, b)\otimes_{\mathbf C[a, b]} M$ is spanned by
$g^J\, (J\in\mathcal F)$ as a linear space over $\mathbf C(a, b)$, 
and we have $\mathrm{rank}(M) \leq |\mathcal F|$.
By this inequality and Lemma \ref{lem:liearly-independent}, 
we have $\mathrm{rank}(M) = |\mathcal F|$.
\end{proof}
\begin{remark}\label{rem:1}
We note that the Heaviside function 
$H\left(\sum_{i=1}^d a_{ij}x_i+b_j\right)$ is equal to 
$\lim_{\lambda\rightarrow 0}\left(\sum_{i=1}^d a_{ij}x_i+b_j\right)_+^\lambda$
as a Schwartz distribution.
Thus we may expect that our integral representation is a ``specialization''
of the integral considered 
in the cohomology groups in \cite{aomoto-kita-orlik-terao}.
However, we have no rigorous understanding of this limiting procedure
of the twisted cohomology.
An interesting observation is that 
the holonomic rank of $M$ can be smaller than the dimension of 
the twisted cohomology group
$H^d\left(\mathcal G, \nabla \right)$,
where 
$
\mathcal G:= 
\mathbf C^d-V\left(\prod_{j=1}^n\left(\sum_{i=1}^d a_{ij}x_i+b_j\right)\right)
$
and 
$\nabla$ is a connection defined in \cite{aomoto-kita-orlik-terao}.
In our case, the dimension of the twisted cohomology is 
$\sum_{i=0}^d\binom{n}{i}$.
The number of faces of the polyhedron in general position with $n$ facets 
can be smaller than this value.
\end{remark}

\medskip\noindent
{\bf Acknowledgements.}\/
The author wishes to thank Nobuki Takayama 
for useful discussions during the preparation of this paper, 
and Akimichi Takemura for the suggestion 
that the holonomic gradient method can be applied 
to the evaluation of the normal probability of convex polyhedra.


\appendix
\section{Proof of Theorem \ref{ptb:8}}
\label{sec:6}

In this appendix, we prove Theorem \ref{ptb:8}.
For this purpose, 
we need to present some notation and some lemmas,
most importantly, Theorem \ref{ptb:5}.
The argument in the proof of Lemma \ref{ptb:1} is also important,
since analogous arguments will appear repeatedly 
in the proof of Theorem \ref{ptb:5}.

Let $a=(a_1, \dots, a_n)\in \mathbf R^{(d+1)\times n} $ be a matrix where
we denote by $a_j$ the $j$-th column vector of $a$.
Put $a_0 :=(1, 0, \dots, 0)^\top\in \mathbf R^{d+1} $.
For a matrix $a$, we put 
\begin{align*}
H_j(a)
&:= \left\{x \in \mathbf R^d :\sum_{i=1}^d a_{ij} x_i +a_{0j} \geq 0 \right\}
\quad (j\in [n]), 
\\
\mathcal H(a)
&:=\left\{H_1(a), \dots, H_n(a)\right\}, 
\\
P(a)
&:= \bigcap_{j=1}^n H_j(a), 
\\
F_j(a)
&:= \partial H_j(a) \cap P(a)\quad (j\in [n]), 
\\
F_J(a)&:= \bigcap_{j\in J}  F_j(a)\quad (J\subset [n]), 
\\
\mathcal F(a)
&:= \left\{J\subset[n]: F_J(a)\neq \emptyset \right\} .
\end{align*}
Note that $F_j(a)$ is not necessarily a facet of $P(a)$, 
and $\mathcal F(a)$ is not necessarily equivalent to 
the abstract simplicial complex associated with $P(a)$.
For this difficulty, 
we need the notion of families of half-spaces in general position.
In fact, in Lemma \ref{ptb:4}, we will show that
the abstract simplicial complex associated with $P(a)$ 
is equivalent to $\mathcal F(a)$ under the general position assumption, 
which is required by the proof of Theorem \ref{ptb:8}.

In order to consider combinatorial structures at the point at infinity, 
we introduce the following notion:
for the abstract simplicial complex $\mathcal F$ associated with a polyhedron, 
the {\it homogenization} of  $\mathcal F$ is the abstract simplicial complex
defined by
\[
\hat{\mathcal F}
:= \left\{J\subset \{0, 1, \dots, n\} : \hat F_J \neq \{0\} \right\}.
\]
Since $\mathcal F\subset \hat{\mathcal F}$, 
we have $\{j\} \in \hat{\mathcal F} $ for $j\in [n]$.
Note that $\{0\} \in \hat{\mathcal F}$ does not hold in general.
The following are the ``homogenization'' of 
the notations given in the previous paragraph.
For a matrix $a$, we put 
\begin{align*}
\hat H_j(a)
&:= \left\{x\in \mathbf R^{d+1}  : \sum_{i=0}^da_{ij} x_i\geq 0\right\}
\quad \left(j\in \{0, 1, \dots, n\}\right), 
\\
\hat{\mathcal H} (a)
&:= \left\{\hat H_j(a) : j\in \{0, 1, \dots, n\} \right\}, 
\\
\hat P(a)
&:= \bigcap_{j=0}^n  \hat H_j(a), 
\\
\hat F_j(a)
&:= \partial \hat H_j(a) \cap \hat P(a)\quad \left(j\in \{0, 1, \dots, n\}\right), 
\\
\hat{\mathcal F} (a)
&:= \left\{J\subset \{0, 1, \dots, n\} : \hat F_J(a) \neq \{0\}  \right\}, 
\\
\hat F_J(a)&:= \bigcap_{j\in J}  \hat F_j(a)\quad 
\left(J\subset \{0, 1, \dots, n\}\right).
\end{align*}
For all $J\subset \{0, 1, \dots, n\}$, $\hat F_J(a)$ includes the zero vector.
Analogous to 
the case of the abstract simplicial complex associated with polyhedra, 
we have $\mathcal F(a) \subset \hat{\mathcal F}(a)$.
For a family of half-spaces in general position, we have the following lemma, 
which is required by the proof of Lemma \ref{ptb:4}.
\begin{lemma} \label{ptb:1}
Suppose $\mathcal H(a)$ is in general position.
Let $J\subset [n]$.
If $J\in \mathcal F(a)$, 
then the set $F_J(a)$ is a $d-|J|$-dimensional face of $P(a)$.
If $J\notin \mathcal F(a)$, 
then the set $F_J(a)$ is empty.
In particular, $P(a)$ is a $d$-dimensional polyhedron.
\end{lemma}
\begin{proof}
If $J\notin \mathcal F(a)$, it is obvious that $F_J(a)=\emptyset$.
Let us consider the case where $J\in \mathcal F(a)$.
Put $J=\{j(0), \dots, j(s)\}, \, J^c:= \{0, 1, \dots, n\}\backslash J$.
Since $\mathcal H(a)$ is in general position, 
we have $\hat F_J(a)=\{0\}$ or
\begin{equation}\label{ptb:12}
\hat F_J(a) \supsetneq \hat F_J(a) \cap \partial \hat H_k(a)
\end{equation}
for any $k\in J^c$.
In fact, 
the equality of \eqref{ptb:12} contradicts the assumption about the dimension
when $\hat F_J(a)\neq\{0\}$.
However, we have $\hat F_J(a)\neq\{0\}$,
since $J\in \mathcal F(a)$ implies $J\in \hat{\mathcal F}(a)$.
Hence, condition \eqref{ptb:12} holds for any $k\in J^c$.
For $k\in J^c$, 
take $x^{(k)} \in \hat F_J(a)\backslash \partial \hat H_k(a)$.
Then the affine combination
\[
\frac{1} {n+1-|J|}  \sum_{k\in J^c}  x^{(k)}
\]
is an element of 
\begin{equation} \label{ptb:2}
\hat F_J(a)
\cap
\left(\bigcap_{k\in J^c}  \mathrm{int} (\hat H_k(a)) \right),
\end{equation}
where  $\mathrm{int} (S)$ denotes the interior of $S$.
Since $0\in J^c$, we have
\[
\hat F_J(a)
\cap \{x_0 =1 \}
\cap \left(\bigcap_{k\in J^c}  \mathrm{int} (\hat H_k(a)) \right)
\neq \emptyset.
\]
Let $x$ be an element of this set, and let 
$B$ be an open ball centered at $x$ whose closure is included in 
$\bigcap_{k\in J^c}  \mathrm{int} \left(\hat H_k(a)\right)$.
Let $x'$ be an arbitrary point in 
$\{x_0 =1 \} \cap\bigcap_{j\in J}  \partial \hat H_j(a) $, and let 
$x''$ be an intersection point of 
$\partial B$ and the line between $x$ and $x'$.
Since $x''\in \hat F_J(a)$ and 
$x'$ can be written as an affine combination of $x$ and $x''$, 
we have $x'\in \aff(\hat F_J(a)).$
By the arbitrariness of $x'$, 
we have 
\[
\aff(F_J(a))
\cong \aff\left(\hat F_J(a)\cap \{x_0=1\} \right)
= \bigcap_{j\in J} \partial \hat H_j(a) \cap \{x_0=1\}.
\]
Analogously, we have 
\begin{equation} \label{ptb:3}
\aff(\hat F_J(a))
=
\bigcap_{j\in J}  \partial \hat H_j(a).
\end{equation}

Since $\hat{\mathcal H} (a)$ is in general position, 
the left-hand side of \eqref{ptb:3} is a $d+1-|J|$-dimensional cone.
Consequently, the vectors $a_{j(1)}, \dots, a_{j(s)} $ are linearly independent.
Moreover, the vectors $a_0, a_{j(1)}, \dots, a_{j(s)}$ are 
linearly independent. 
In fact, if there are $c_j\, (j\in J)$ such that $a_0 =\sum_{j\in J} c_ja_j $, 
then we have 
$
0=\sum_{j\in J} c_j \left(\sum_{i=1}^d a_{ij} x_i+a_{0j} \right)=1
$
for $x\in F_J.$
Hence, the dimension of $\aff(F_J)$ is equal to $d-|J|$.
\end{proof}
By the argument in the proof of Lemma \ref{ptb:1}, 
when $\mathcal H(a)$ is in general position, 
$J\in \hat{\mathcal F}(a)$ implies $F_J(a)\neq\emptyset$.
Hence, we have the following corollary.
\begin{corollary}\label{ptb:11}
When $\mathcal H(a)$ is in general position, 
for $J\subset[n]$, 
$J\in \mathcal F(a)$ holds if and only if $J\in \hat{\mathcal F}(a)$ holds.
\end{corollary}
Similarly, by the argument in the proof of Lemma \ref{ptb:1}, 
we have the following.
\begin{corollary}\label{ptb:13}
When $\mathcal H(a)$ is in general position and 
$J=\{j(1), \dots, j(s)\}\in \mathcal F(a)$, 
the column vectors $a_{j(1)}, \dots, a_{j(s)}$ are linearly independent.
\end{corollary}

We now review the Farkas lemma in \cite{ziegler}.
\begin{proposition}[Farkas lemma] \label{ptb:10}
The inequality $\sum_{i=1}^dc_ix_i +c_0\geq 0$ is valid for $P(a)$
if and only if one of the following conditions holds:
\begin{enumerate}
  \renewcommand{\labelenumi}{(\roman{enumi})}         
\item There exist $\lambda_j\geq 0\, (j\in [n])$ such that 
\[
\sum_{j=1}^n a_{ij} \lambda_j = c_i, \quad
\sum_{j=1}^n a_{0j} \lambda_j \leq c_0\quad (1\leq i\leq d);
\]
\item There exist $\lambda_j\geq 0\, (j\in [n])$ such that
\[
\sum_{j=1}^n a_{ij} \lambda_j = 0, \quad
\sum_{j=1}^n a_{0j} \lambda_j < 0\quad (1\leq i\leq d).
\]
\end{enumerate}
\end{proposition}
Note that condition (ii) in the Farkas lemma implies $P(a)=\emptyset$.

Lemma \ref{ptb:1} and Proposition \ref{ptb:10} imply the following lemma, 
which is not required by the proof of Theorem \ref{ptb:5} but is required  
by that of Theorem \ref{ptb:8}.
\begin{lemma} \label{ptb:4}
If $\mathcal H(a)$ is in general position, 
then the abstract simplicial complex associated with $P(a)$ is equivalent to 
$\mathcal F(a)$.
\end{lemma}
\begin{proof}
Note that we assume $d\geq 1$
and $P(a)$ is a $d$-dimensional polyhedron.
Let $F$ be a facet of $P(a)$; 
then there is an inequality $\sum_{i=1}^d c_ix_i + c_0\geq 0$ valid for $P(a)$
such that 
\[
F = P(a) \cap \left\{x\in \mathbf R^d : \sum_{i=1}^d c_ix_i + c_0= 0\right\}.
\]
Since $P(a)$ is not empty, 
condition (ii) in the Farkas lemma does not hold.
By the Farkas lemma, 
there exist $\lambda_j\geq 0\, (j\in [n])$ such that 
\[
\sum_j a_{ij} \lambda_j = c_i, \quad
\sum_j a_{0j} \lambda_j \leq c_0\quad (1\leq i\leq d).
\]
Moreover, there is a unique $\lambda_j$ which is greater than $0$.
In fact, if there is not such a $j$, then we have 
$c_i=0\, (1\leq i\leq d)$. This implies $F=P(a)$ or $F=\emptyset$.
This is a contradiction.

If there are two indexes $j(1)$ and $j(2)$ such that 
$\lambda_{j(1)} >0$,  $\lambda_{j(2)} >0$, and $j(1)\neq j(2)$, 
then we have
$
F \subset P(a)
\cap \partial H_{j(1)} (a) \cap \partial H_{j(2)} (a).
$
This contradicts the fact that $F$ is $d-|J|$-dimensional.

Therefore, there is a unique $j$ such that 
$a_{ij} \lambda_j=c_i, \, a_{0j} \lambda_j\leq c_0.$
Since we have 
\[
0 \leq \lambda_j\left(\sum_{i=1}^da_{ij}x_i+a_{0j}\right)\leq 
\sum_{i=1}^dc_ix_i+c_0= 0
\]
for $x\in F$, we have $c_0= a_{0j} \lambda_j$.
Consequently, $F=F_j(a)$.

If $F_j(a)\neq \emptyset$, then, by Lemma \ref{ptb:1}, $F_j(a)$ is a facet of $P(a)$.

Therefore, all the facets of $P(a)$ are given by 
$F_j(a)\, (\{j\} \in \mathcal F(a))$.
Consequently, the abstract simplicial complex $\mathcal F(a)$ associated with 
$P(a)$ is equivalent to $\mathcal F(a)$.
\end{proof}

\begin{theorem} \label{ptb:5}
Let $P\subset \mathbf R^d$ be a polyhedron in general position, let 
$n$ be the number of facets of $P$, 
and let $\mathcal F$ be the abstract simplicial complex associated with $P$.
Then, the set 
\[
U =\left\{
a\in \mathbf R^{(d+1)\times n}  :
\text{$\mathcal H(a)$ is in general position 
 and $\hat{\mathcal F} (a) = \hat{\mathcal F} $}
\right\}
\]
is an open set of $\mathbf R^{(d+1)\times n}$.
\end{theorem}

\begin{proof}
Our proof is given by writing $U$ as an intersection of 
finite open sets $V_J\, (J\subset\{0, 1, \dots, n\})$: 
\begin{equation} \label{ptb:6}
U = \left(\bigcap_{J\subset \{0, 1, \dots, n\}}  V_J\right).
\end{equation}

\medskip\noindent
{\bf STEP 1} \/.
First, we define $V_J$.
For $J=\{j(1), \dots, j(s)\} \in \hat{\mathcal F} $, 
let $V_J$ consist of $a\in \mathbf R^{(d+1)\times n} $ such that
the vectors $a_{j(1)}, \dots, a_{j(s)} $ are linearly independent
and the set \eqref{ptb:2} includes a nonzero element.
For $J\subset \{0, 1, \dots, n\}$ such that $J\notin \hat{\mathcal F} $, 
let $V_J$ be the intersection of 
\[
\left\{a\in \mathbf R^{(d+1)\times n}  : \hat F_J(a)=\{0\} \right\}
\]
and 
\begin{equation}\label{ptb:9}
\bigcap_{j=1}^n
\left\{a\in \mathbf R^{(d+1)\times n}  : a_j\neq 0\right\}.
\end{equation}

\medskip\noindent
{\bf STEP 2} \/.
Next, we show that the right-hand side of \eqref{ptb:6} includes 
the left-hand side.
Let $J=\{j(1), \dots, j(s)\} \subset \{0, 1, \dots, n\}$.

Suppose $J\in \hat{\mathcal F} $.
Let $a\in U$; then we have $\hat{\mathcal F} (a)=\hat{\mathcal F}$ 
by the definition of $U$.
When $J^c=\emptyset$, 
the set \eqref{ptb:2} is equal to $\hat F_J(a)$ and includes a nonzero element.
When $J^c\neq\emptyset$, 
by an argument analogous to that in the proof of Lemma \ref{ptb:1}, 
we can show that the set \eqref{ptb:2} includes a nonzero element.
Consequently, equation \eqref{ptb:3} holds.
Since the left-hand side of \eqref{ptb:3} is 
a $d+1-|J|$-dimensional affine subspace, 
the vectors $a_{j(1)}, \dots, a_{j(s)} $ are linearly independent. 
Hence, we have $U\subset V_J\, (J\in \mathcal F)$.

Suppose $J\notin\hat{\mathcal F}$.
For $a\in U$, it is obvious that $\hat F_J(a)=\{0\}$.
By the assumption for $\mathcal F$, 
we have $\{j\} \in \hat{\mathcal F} $ for arbitrary $j\in [n]$.
By the argument in the previous paragraph, $a\in U$ implies $a_j\neq 0$.
Hence, we have $U\subset V_J\, (J\notin\mathcal F)$.

Therefore, the right-hand side of \eqref{ptb:6} includes the left-hand side.

\medskip\noindent
{\bf STEP 3} \/.
We show that the left-hand side of \eqref{ptb:6} includes the right-hand side. 
Suppose a matrix  $a$ is included in the right-hand side of \eqref{ptb:6}.
Let $J\subset [n]$.
When $J\in \mathcal F$, 
$a \in V_J$ implies that the set \eqref{ptb:2} includes a nonzero element.
Hence, we have $\hat F_J(a) \neq \{0\} $.
When $J\notin \mathcal F$, 
$a\in V_J$ implies $\hat F_J(a) = \{0\} $.
Consequently, we have $\hat{\mathcal F}(a)=\hat{\mathcal F}$.

Next,  we show that $\mathcal H(a)$ is in general position.
It is enough to show that 
$\hat F_J(a)$ is a $d+1-|J|$-dimensional cone or that it is equal to $\{0\} $
for $J=\{j(1), \dots, j(s)\} \subset \{0, 1, \dots, n\}$.
When $J\in \hat{\mathcal F} $, 
set \eqref{ptb:2} is not empty and equation \eqref{ptb:3} holds.
Since the vectors $a_{j(1)}, \dots, a_{j(s)} $ are linearly independent, 
the right-hand side of \eqref{ptb:3} is a $d+1-|J|$-dimensional affine subspace.
When $J\not\in \hat{\mathcal F} $, we have $\hat F_J(a)=\{0\} $.
Hence, $\mathcal H(a)$ is in general position.
Therefore, equation \eqref{ptb:6} holds. 

\medskip\noindent
{\bf STEP 4} \/.
Let $J\subset \{0, 1, \dots, n\}$.
We show that $V_J$ is an open set.

Suppose $J\in\mathcal F$.
Let $a\in V_J$. 
Since the vectors $a_j (j\in J)$ are linearly independent, 
there is a subset $I\subset[n]$ such that $|I|=|J|$ and 
the $|I|\times|J|$ matrix $(a_{ij} )_{i\in I, j\in J} $ has 
the inverse matrix $(c_{ji} )_{i\in I, j\in J} $.
Let $\sigma:I\rightarrow J$ be a bijection, then 
\[
y_k =
\begin{cases}
\sum_{i\in I} a_{i\sigma(k)} x_i & (k \in I)\\
x_k & (k\in I^c)
\end{cases}
\]
defines a coordinate transformation.
By this coordinate transformation, 
the condition that set \eqref{ptb:2} includes a nonzero element
can be written as follows: 
there exists $y\in \mathbf R^d$ such
that  
\begin{equation} \label{ptb:7}
\begin{cases}
y_{\sigma^{-1} (j)} +\sum_{i\in I^c}  y_ia_{ij}  =0 
& (j\in J), \\
\sum_{i\in I} \sum_{k\in J}  y_{\sigma^{-1} (k)} c_{ki} a_{ij} +\sum_{i\in I^c}  a_{ij} y_i>0
& (j\in J^c), \\
y\neq 0 .
\end{cases}
\end{equation}

Let $V_{IJ} $ be the set consisting of $a$ such that 
the matrix $(a_{ij} )_{i\in I, j\in J} $ is invertible and 
there exists $y$ satisfying condition \eqref{ptb:7}.
Then, we have 
\[
V_J= \bigcup_{\substack{I\subset[n], \\ |I|=|J|} }  V_{IJ}.
\]
In equation \eqref{ptb:7}, 
the value of $y_j\, (j\in J)$ is determined uniquely 
by the value of the other variables.
We can regard \eqref{ptb:7} as a condition for the variables 
$a, \, y_j (j\notin J)$.
The set $V_{IJ} $ is the projection of the following open subset 
of $\mathbf R^{(d+1)\times n} \times \mathbf R^{n-|J|}$: 
the set consisting of 
$(a, y_j)_{j\notin J} \in \mathbf R^{(d+1)\times n} \times \mathbf R^{n-|J|}$
that satisfies \eqref{ptb:7} and in which 
$(a_{ij} )_{i\in I, j\in J} $ is invertible.
This implies that $V_{IJ}$ is open.
Consequently, $V_J$ is an open set.

Suppose $J\notin\mathcal F$.
There is the isomorphism 
between the set \eqref{ptb:9} and $\mathbf R_{>0}^n\times(S^d)^n$.
Here, $(S^d)^n$ denotes the $n$-th direct product of a $d$-dimensional sphere $S^d$.
Since $V_J$ is closed under the transformation 
which multiplies column vectors by positive numbers,
\[
(a_{ij} ) \mapsto (\lambda_ja_{ij} ), \quad (\lambda_j \in \mathbf R_{>0} ), 
\]
the image of $V_J$ under the isomorphism is 
$\mathbf R_{>0}^n \times \left((S^d)^n\cap V_J\right)$.
It is enough to show $(S^d)^n\cap V_J$ is an open set.
Since the projection
\[
(S^d)^n\times S^d \rightarrow (S^d)^n, \quad (a, x) \mapsto a
\]
is a continuous map from a compact set to a Hausdorff space, it is a closed map.
The set $(S^d)^n\backslash V_J$ is the projection of the following closed set 
of $(S^d)^n\times S^d$:
the intersection of $(S^d)^n\times S^d$ 
and the subset of $\mathbf R^{(d+1)\times n}\times \mathbf R^{d+1}$ 
consisting of $(a,x)$ satisfying 
\[
\begin{cases}
\sum_{i=0}^d a_{ij}x_i = 0 & (j\in J)\\
\sum_{i=0}^d a_{ij}x_i \geq 0 & (j\in J^c).
\end{cases}
\]
Hence, $(S^d)^n\cap V_J$ is open.
\end{proof}
\begin{proof}[Proof of Theorem \ref{ptb:8}]
With the notation of Theorem \ref{ptb:5}, 
it is enough to show that for $a\in U$, 
the abstract simplicial complex associated with $P(a)$ is equivalent to 
$\mathcal F$.
Let $a\in U$.
By Lemma \ref{ptb:4}, 
the abstract simplicial complex associated with $P(a)$
is equivalent to $\mathcal F(a)$.
By Corollary \ref{ptb:11},
$\hat{\mathcal F}(a)=\hat{\mathcal F}$  implies $\mathcal F(a)=\mathcal F$.
\end{proof}

\end{document}